\documentclass[12pt]{article}
\usepackage{amsfonts}
\usepackage{amsmath}
\usepackage{amssymb}
\usepackage{amsthm}
\newtheorem{theorem}{Theorem}[section]
\newtheorem{lemma}{Lemma}[section]

\theoremstyle{definition}
\newtheorem{definition}{Definition}[section]

\begin{document}

\begin{center}
\vskip 1cm{\LARGE\bf On the Density of Ranges of Generalized Divisor Functions  
\vskip 1cm
\large
Colin Defant\footnote{This work was supported by National Science Foundation grant no. 1262930.}\\
Department of Mathematics\footnote{1400 Stadium Rd \\
University of Florida \\
Gainesville, FL 32611}\\
University of Florida\\
United States\\
cdefant@ufl.edu}
\end{center}
\vskip .2 in

\begin{abstract}
The range of the divisor function $\sigma_{-1}$ is dense in the interval $[1,\infty)$. However, the range of the function $\sigma_{-2}$ is not dense in the interval $\displaystyle{\left[1,\frac{\pi^2}{6}\right)}$. We begin by generalizing the divisor functions to a class of functions $\sigma_{t}$ for all real $t$. We then define a constant $\eta\approx 1.8877909$ and show that if $r\in(1,\infty)$, then the range of the function $\sigma_{-r}$ is dense in the interval $[1,\zeta(r))$ if and only if $r\leq\eta$. We end with an open problem.     
\end{abstract} 

\section{Introduction} 
Throughout this paper, we will let $\mathbb{N}$ denote the set of positive integers, and we will let $p_i$ denote the $i^{th}$ prime number.  
\par 
For any integer $t$, the divisor function $\sigma_t$ is a multiplicative arithmetic function defined by $\displaystyle{\sigma_t(n)=\sum_{\substack{d\vert n \\ d>0}}d^t}$ for all positive integers $n$. The value of $\sigma_1(n)$ is the sum of the positive divisors of $n$, while the value of $\sigma_0(n)$ is simply the number of positive divisors of $n$. Another interesting divisor function is $\sigma_{-1}$, which is often known as the abundancy 
index. One may show \cite{Laatsch86} that the range of $\sigma_{-1}$ is a subset of the interval $[1,\infty)$ that is dense in $[1,\infty)$. If $t<-1$, then the range of $\sigma_t$ is a subset of the interval $[1,\zeta(-t))$, where $\zeta$ denotes the 
Riemann zeta function. This is because, for any positive integer $n$, $\displaystyle{\sigma_{t}(n)=\sum_{\substack{d\vert n \\ d>0}}d^t<\sum_{i=1}^{\infty}i^t=\zeta(-t)}$. For example, the range of the function $\sigma_{-2}$ is a subset of the interval $\displaystyle{\left[1,\frac{\pi^2}{6}\right)}$. However, it is interesting to note that the range of the 
function $\sigma_{-2}$ is not dense in the interval $\displaystyle{\left[1,\frac{\pi^2}{6}\right)}$. To see this, let $n$ be a positive integer. If $2\vert n$, then $\displaystyle{\sigma_{-2}(n)\geq \frac{1}{1^2}+\frac{1}{2^2}=\frac{5}{4}}$. On the other hand, if $2\nmid n$, then $\displaystyle{\sigma_{-2}(n)<\sum_{d\in\mathbb{N}\backslash(2\mathbb{N})}\frac{1}{d^2}=\frac{\zeta(2)}{\left(\frac{1}{1-2^{-2}}\right)}=\frac{\pi^2}{8}}$. As $\displaystyle{\frac{\pi^2}{8}<\frac{5}{4}}$, we see that there is a ``gap" in the range of $\sigma_{-2}$. In other words, there are no positive integers $n$ such that $\displaystyle{\sigma_{-2}(n)\in\left(\frac{\pi^2}{8},\frac{5}{4}\right)}$. 
\par 
Our first goal is to generalize the divisor functions to allow for nonintegral subscripts. For example, we might consider the function $\sigma_{-\sqrt{2}}$, defined by $\displaystyle{\sigma_{-\sqrt{2}}(n)=\sum_{\substack{d\vert n \\ d>0}}d^{-\sqrt{2}}}$. We formalize this idea in the following definition. 
\begin{definition} \label{Def1.1} 
For a real number $t$, define the function $\sigma_t\colon\mathbb{N}\rightarrow\mathbb{R}$ by $\displaystyle{\sigma_t(n)=\sum_{\substack{d\vert n \\ d>0}}d^t}$ 
for all $n\in\mathbb{N}$. Also, we will let $\log\sigma_t=\log\circ\hspace{0.5 mm}\sigma_t$. 
\end{definition} 
\par 
In analyzing the ranges of these generalized divisor functions, we will find a constant which serves as a ``boundary" between divisor functions with dense ranges and divisor functions with ranges that have gaps. Note that, for any real number $t$, we may write $\sigma_t=I_0*I_t$, where $I_0$ and $I_t$ are arithmetic functions defined by $I_0(n)=1$ and $I_t(n)=n^t$. As $I_0$ and $I_t$ are multiplicative, we find that $\sigma_t$ is multiplicative. 

\section{The Ranges of Functions $\sigma_{-r}$} 

\begin{theorem} \label{Thm2.1}
Let $r$ be a real number greater than $1$. The range of $\sigma_{-r}$ is dense in the interval $[1,\zeta(r))$ if and only if $\displaystyle{1+\frac{1}{p_m^r}\leq\prod_{i=m+1}^{\infty}\left(\sum_{j=0}^{\infty}\frac{1}{p_i^{jr}}\right)}$ for all positive integers $m$. 
\end{theorem} 
\begin{proof}
First, suppose that $\displaystyle{1+\frac{1}{p_m^r}\leq\prod_{i=m+1}^{\infty}\left(\sum_{j=0}^{\infty}\frac{1}{p_i^{jr}}\right)}$ for all positive integers $m$. We will show that the range of $\log\sigma_{-r}$ is dense in the interval $[0,\log(\zeta(r)))$, which will imply that the range of $\sigma_{-r}$ is dense in $[1,\zeta(r))$. Choose some arbitrary $x\in(0,\log(\zeta(r)))$, and define $X_0=0$. For each positive integer $n$, we define $\alpha_n$ and $X_n$ in the following manner. If \\ 
$\displaystyle{X_{n-1}+\log\left(\sum_{j=0}^{\infty}\frac{1}{p_n^{jr}}\right)\leq x}$, define $\alpha_n=-1$. If $\displaystyle{X_{n-1}+\log\left(\sum_{j=0}^{\infty}\frac{1}{p_n^{jr}}\right)>x}$, define $\alpha_n$ to be the largest nonnegative integer that satisfies \\ 
$\displaystyle{{X_{n-1}+\log\left(\sum_{j=0}^{\alpha_n}\frac{1}{p_n^{jr}}\right)\leq x}}$. Define $X_n$ by 
\[X_n=\begin{cases} X_{n-1}+\log\left(\sum_{j=0}^{\alpha_n}\frac{1}{p_n^{jr}}\right), & \mbox{if } \alpha_n\geq 0; \\ X_{n-1}+\log\left(\sum_{j=0}^{\infty}\frac{1}{p_n^{jr}}\right), & \mbox{if } \alpha_n=-1. \end{cases}\]
 Also, for each $n\in\mathbb{N}$, define $D_n$ by 
 \[D_n=\begin{cases} \log\left(\sum_{j=0}^{\infty}\frac{1}{p_n^{jr}}\right)-\log\left(\sum_{j=0}^{\alpha_n}\frac{1}{p_n^{jr}}\right), & \mbox{if } \alpha_n\geq 0; \\ 0, & \mbox{if } \alpha_n=-1, \end{cases}\]
 and let $E_n=\displaystyle{\sum_{i=1}^n D_i}$. Note that 
\[\lim_{n\rightarrow\infty}(X_n+E_n)=\lim_{n\rightarrow\infty}\left(X_n+\sum_{i=1}^nD_i\right)\]
\[=\lim_{n\rightarrow\infty}\sum_{i=1}^n\log\left(\sum_{j=0}^{\infty}\frac{1}{p_i^{jr}}\right)=\log(\zeta(r)).\] 
\par 
Now, because the sequence $(X_n)_{n=1}^{\infty}$ is bounded and monotonic, we know that there exists some real number $\gamma$ such that $\displaystyle{\lim_{n\rightarrow\infty}}X_n=\gamma$. We wish to show that $\gamma=x$. 
\par 
Notice that we defined the sequence $(X_n)_{n=1}^{\infty}$ so that $X_n\leq x$ for all $n\in\mathbb{N}$. Hence, we know that $\gamma\leq x$. Now, suppose $\gamma<x$. Then $\displaystyle{\lim_{n\rightarrow\infty}}E_n=\log(\zeta(r))-\gamma>\log(\zeta(r))-x$. This implies that there 
exists some positive integer $N$ such that $E_n>\log(\zeta(r))-x$ for all integers $n\geq N$. Let $m$ be the smallest positive integer that satisfies $E_m>\log(\zeta(r))-x$. If $\alpha_m=-1$ and $m>1$, then $D_m=0$, so $E_{m-1}=E_m>\log(\zeta(r))-x$. However, this contradicts the minimality of $m$. If $\alpha_m=-1$ and $m=1$, then $0=D_m=E_m>\log(\zeta(r))-x$, which is also a contradiction. Thus, we conclude that $\alpha_m\geq 0$. This means that $\displaystyle{X_m+D_m=X_{m-1}+\log\left(\sum_{j=0}^{\infty}\frac{1}{p_m^{jr}}\right)>x}$, so $D_m>x-X_m$. Furthermore, 
\[\log\left(\prod_{i=m+1}^{\infty}\left(\sum_{j=0}^{\infty}\frac{1}{p_i^{jr}}\right)\right)=\sum_{i=m+1}^{\infty}\log\left(\sum_{j=0}^{\infty}\frac{1}{p_i^{jr}}\right)\]
\[=\log(\zeta(r))-\sum_{i=1}^m\log\left(\sum_{j=0}^{\infty}\frac{1}{p_i^{jr}}\right)\] 
\begin{equation} \label{Eq2.1}
=\log(\zeta(r))-E_m-X_m<x-X_m<D_m,
\end{equation} 
and we originally assumed that $\displaystyle{1+\frac{1}{p_m^r}\leq\prod_{i=m+1}^{\infty}\left(\sum_{j=0}^{\infty}\frac{1}{p_i^{jr}}\right)}$.     
This means that $\displaystyle{\log\left(1+\frac{1}{p_m^r}\right)<D_m=\log\left(\sum_{j=0}^{\infty}\frac{1}{p_m^{jr}}\right)-\log\left(\sum_{j=0}^{\alpha_m}\frac{1}{p_m^{jr}}\right)}$, or, \\ equivalently, $\displaystyle{\log\left(1+\frac{1}{p_m^r}\right)+\log\left(\sum_{j=0}^{\alpha_m}\frac{1}{p_m^{jr}}\right)<\log\left(\frac{p_m^r}{p_m^r-1}\right)}$. If $\alpha_m>0$, we have 
\[\log\left(\left(1+\frac{1}{p_m^r}\right)^2\right)\leq\log\left(1+\frac{1}{p_m^r}\right)+\log\left(\sum_{j=0}^{\alpha_m}\frac{1}{p_m^{jr}}\right)<\log\left(\frac{p_m^r}{p_m^r-1}\right),\] 
so $\displaystyle{\left(1+\frac{1}{p_m^r}\right)^2<\frac{p_m^r}{p_m^r-1}}$. We may write this as $\displaystyle{1+\frac{2}{p_m^r}+\frac{1}{p_m^{2r}}<1+\frac{1}
{p_m^r-1}}$, so $\displaystyle{2<\frac{p_m^r}{p_m^r-1}=1+\frac{1}{p_m^r-1}}$. As $p_m^r>2$, this is a contradiction. 
Hence, $\alpha_m=0$. By the definitions of $\alpha_m$ and $X_m$, this implies that \\ 
$\displaystyle{X_{m-1}+\log\left(1+\frac{1}{p_m^r}\right)>x}$ and that $X_m=X_{m-1}$. Therefore, \\ 
$\displaystyle{\log\left(1+\frac{1}{p_m^r}\right)>x-X_{m-1}=x-X_m}$.
However, recalling from \eqref{Eq2.1} that 
\[\sum_{i=m+1}^{\infty}\log\left(\sum_{j=0}^{\infty}\frac{1}{p_i^{jr}}\right)<x-X_m,\] we find that 
\[\displaystyle{\sum_{i=m+1}^{\infty}\log\left(\sum_{j=0}^{\infty}\frac{1}{p_i^{jr}}\right)<\log\left(1+\frac{1}{p_m^r}\right)},\] which is a contradiction because we originally assumed that \\ 
$\displaystyle{1+\frac{1}{p_m^r}\leq\prod_{i=m+1}^{\infty}\left(\sum_{j=0}^{\infty}\frac{1}{p_i^{jr}}\right)}$. Therefore, $\gamma=x$. 
\par 
We now know that $\displaystyle{\lim_{n\rightarrow\infty}X_n}=x$. To show that the range of $\log\sigma_{-r}$ is dense in $[0,\log(\zeta(r)))$, we need to construct a sequence $(C_n)_{n=1}^{\infty}$ of elements of the range of $\log\sigma_{-r}$ that satisfies $\displaystyle{\lim_{n\rightarrow\infty}C_n}=x$. We do so in the following fashion. For each positive integer $n$, write 
\[Y_n=\begin{cases} 1, & \mbox{if } \alpha_n\geq 0; \\ 0, & \mbox{if } \alpha_n=-1, \end{cases}\]  
\[Z_n=\begin{cases} 0, & \mbox{if } \alpha_n\geq 0; \\ 1, & \mbox{if } \alpha_n=-1, \end{cases}\]
and 
\[\beta_n=\begin{cases} \alpha_n, & \mbox{if } \alpha_n\geq 0; \\ 0, & \mbox{if } \alpha_n=-1. \end{cases}\]
Now, for each positive integer $n$, define $C_n$ by 
\[C_n=\sum_{k=1}^n\left(Y_k\log\left(\sum_{j=0}^{\beta_k}\frac{1}{p_k^{jr}}\right)+Z_k\log\left(\sum_{j=0}^n\frac{1}{p_k^{jr}}\right)\right).\]
Notice that, by the way we defined $X_n$, we have 
\[X_n=\sum_{k=1}^n\left(Y_k\log\left(\sum_{j=0}^{\beta_k}\frac{1}{p_k^{jr}}\right)+Z_k\log\left(\sum_{j=0}^{\infty}\frac{1}{p_k^{jr}}\right)\right).\]
Therefore, $\displaystyle{\lim_{n\rightarrow\infty}C_n=\lim_{n\rightarrow\infty}X_n=x}$.  All we need to do now is show that each $C_n$ is in the range of $\log\sigma_{-r}$. We have 
\[C_n=\sum_{k=1}^n\left(Y_k\log\left(\sum_{j=0}^{\beta_k}\frac{1}{p_k^{jr}}\right)+Z_k\log\left(\sum_{j=0}^n\frac{1}{p_k^{jr}}\right)\right)\]
\[=\sum_{\substack{k\in\mathbb{N} \\ k\leq n \\ \alpha_k\geq 0}}\log\left(\sum_{j=0}^{\alpha_k}\frac{1}{p_k^{jr}}\right)+\sum_{\substack{k\in\mathbb{N} \\ k\leq n \\ \alpha_k=-1}}\log\left(\sum_{j=0}^n\frac{1}{p_k^{jr}}\right)\] 
\[=\log\left[\left(\prod_{\substack{k\in\mathbb{N} \\ k\leq n \\ \alpha_k\geq 0}}\sigma_{-r}(p_k^{\alpha_k})\right)\left(\prod_{\substack{k\in\mathbb{N} \\ k\leq n \\ \alpha_k=-1}}\sigma_{-r}(p_k^n)\right)\right]\]
\[=\log\sigma_{-r}\left(\left(\prod_{\substack{k\in\mathbb{N} \\ k\leq n \\ \alpha_k\geq 0}}p_k^{\alpha_k}\right)\left(\prod_{\substack{k\in\mathbb{N} \\ k\leq n \\ \alpha_k\geq 0}}p_k^n\right)\right).\]
We finally conclude that if $\displaystyle{1+\frac{1}{p_m^r}\leq\prod_{i=m+1}^{\infty}\left(\sum_{j=0}^{\infty}\frac{1}{p_i^{jr}}\right)}$ for all positive integers $m$, then the range of $\sigma_{-r}$ is dense in the interval $[1,\zeta(r))$. 
\par 
Conversely, suppose that there exists some positive integer $m$ such that \\ 
$\displaystyle{1+\frac{1}{p_m^r}>\prod_{i=m+1}^{\infty}\left(\sum_{j=0}^{\infty}\frac{1}{p_i^{jr}}\right)}$. 
Fix some $N\in\mathbb{N}$, and let $\displaystyle{N=\prod_{i=1}^v q_i^{\gamma_i}}$ be the canonical prime factorization of $N$. If $p_s\vert N$ for some $s\in\{1,2,\ldots,m\}$, then 
\[\sigma_{-r}(N)\geq 1+\frac{1}{p_s^r}\geq 1+\frac{1}{p_m^r}.\]
On the other hand, if $p_s\nmid N$ for all $s\in\{1,2,\ldots,m\}$, then 
\[\sigma_{-r}(N)=\prod_{i=1}^v \sigma_{-r}(q_i^{\gamma_i})=\prod_{i=1}^v \left(\sum_{j=0}^{\gamma_i}\frac{1}{q_i^{jr}}\right)\] 
\[<\prod_{i=1}^v \left(\sum_{j=0}^{\infty}\frac{1}{q_i^{jr}}\right)<\prod_{i=m+1}^{\infty}\left(\sum_{j=0}^{\infty}\frac{1}{p_i^{jr}}\right).\] 
Because $N$ was arbitrary, this shows that there is no element of the range of $\sigma_{-r}$ in the interval $\displaystyle{\left[\prod_{i=m+1}^{\infty}\left(\sum_{j=0}^{\infty}\frac{1}{p_i^{jr}}\right),1+\frac{1}{p_m^r}\right)}$, which means that the range of $\sigma_{-r}$ is not dense in $[1,\zeta(r))$.  
\end{proof} 
Theorem \ref{Thm2.1} provides us with a method to determine values of $r>1$ with the property that the range of $\sigma_{-r}$ is dense in $[1,\zeta(r))$. However, doing so is still a somewhat difficult task. Luckily, for $r\in(1,2]$, we may greatly simplify the problem with the help of the following theorem. First, we need a short lemma.  
\begin{lemma} \label{Lem2.1}
If $j\in\mathbb{N}\backslash\{1,2,4\}$, then $\displaystyle{\frac{p_{j+1}}{p_j}<\sqrt{2}}$. 
\end{lemma}
\begin{proof}
Pierre Dusart \cite{Dusart10} has shown that, for $x\geq 396\hspace{0.75 mm} 738$, there must be at  least one prime in the interval $\displaystyle{\left[x, x+\frac{x}{25\log^2x}\right]}$. Therefore, whenever $p_j>396\hspace{0.75 mm} 738$, we may set $x=p_j+1$ to get $\displaystyle{p_{j+1}\leq (p_j+1)+\frac{p_j+1}{25\log^2(p_j+1)}}$ $<\sqrt{2}p_j$. Using Mathematica 9.0 \cite{Wolfram09}, we may quickly search through all the primes less than $396\hspace{0.75 mm} 738$ to conclude the desired result.   
\end{proof} 
\begin{theorem} \label{Thm2.2} 
Let $r$ be a real number in the interval $(1,2]$. The range of $\sigma_{-r}$ is dense in the interval $[1,\zeta(r))$ if and only if $\displaystyle{1+\frac{1}{p_m^r}\leq\prod_{i=m+1}^{\infty}\left(\sum_{j=0}^{\infty}\frac{1}{p_i^{jr}}\right)}$ for all $m\in\{1,2,4\}$. 
\end{theorem}
\begin{proof}
Let $\displaystyle{F(m,r)=\left(1+\frac{1}{p_m^r}\right)\prod_{i=1}^m\left(\sum_{j=0}^{\infty}\frac{1}{p_i^{jr}}\right)}$ so that the inequality \\ 
$\displaystyle{1+\frac{1}{p_m^r}\leq\prod_{i=m+1}^{\infty}\left(\sum_{j=0}^{\infty}\frac{1}{p_i^{jr}}\right)}$ is equivalent to $F(m,r)\leq\zeta(r)$. In light of Theorem \ref{Thm2.1}, it suffices to show that if $F(m,r)\leq\zeta(r)$ for all $m\in\{1,2,4\}$, then $F(m,r)\leq\zeta(r)$ for all $m\in\mathbb{N}$. Thus, let us assume that $r$ is such that $F(m,r)\leq\zeta(r)$ for all $m\in\{1,2,4\}$. If $m\in\mathbb{N}\backslash\{1,2,4\}$, then Lemma \ref{Lem2.1} tells us that $\displaystyle{\frac{p_{m+1}}{p_m}<\sqrt{2}\leq\sqrt[r]{2}}$, which implies that $\displaystyle{\frac{2}{p_{m+1}^r}>\frac{1}{p_m^r}}$. We then have 
\[F(m+1,r)=\left(1+\frac{1}{p_{m+1}^r}\right)\prod_{i=1}^{m+1}\left(\sum_{j=0}^{\infty}\frac{1}{p_i^{jr}}\right)>\left(1+\frac{1}
{p_{m+1}^r}\right)^2\prod_{i=1}^m\left(\sum_{j=0}^{\infty}\frac{1}{p_i^{jr}}\right)\] 
\[>\left(1+\frac{2}{p_{m+1}^r}\right)\prod_{i=1}^m\left(\sum_{j=0}^{\infty}\frac{1}{p_i^{jr}}\right)>\left(1+\frac{1}{p_m^r}\right)\prod_{i=1}^m\left(\sum_{j=0}^{\infty}\frac{1}{p_i^{jr}}\right)=F(m,r)\]
for all $m\in\mathbb{N}\backslash\{1,2,4\}$. This means that $F(3,r)<F(4,r)\leq\zeta(r)$. Furthermore, $F(m,r)<\zeta(r)$ for all integers $m\geq 5$ because $\displaystyle{(F(m,r))_{m=5}^{\infty}}$
is a strictly increasing sequence and $\displaystyle{\lim_{m\rightarrow\infty}F(m,r)=\zeta(r)}$.  
\end{proof}
We have seen that, for $r\in(1,2]$, the range of $\sigma_{-r}$ is dense in $[1,\zeta(r))$ if and only if $F(m,r)\leq\zeta(r)$ for all $m\in\{1,2,4\}$. Using Mathematica 9.0, one may plot a 
function $g_m(r)=F(m,r)-\zeta(r)$ for each $m\in\{1,2,4\}$. It is then easy to verify that $g_2$ has precisely one root, say $\eta$, in the interval $(1,2]$ (for anyone seeking a more rigorous proof of this fact, we mention that it is fairly simple to show that $g_2'(r)>0$ for all $r\in(1,2]$). Furthermore, one may confirm that $g_1(r),g_2(r),g_4(r)\leq 0$ for all $r\in(1,\eta]$ and that $g_2(r)>0$ for all $r\in(\eta,3]$. Hence, we have proven (or at least left the reader to verify) the first part of the following theorem.  
\begin{theorem} \label{Thm2.3} 
Let $\eta$ be the unique number in the interval $(1,2]$ that satisfies the equation 
\[\left(\frac{2^{\eta}}{2^{\eta}-1}\right)\left(\frac{3^{\eta}+1}{3^{\eta}-1}\right)=\zeta(\eta).\] If $r\in(1,\infty)$, then the range of the function $\sigma_{-r}$ is dense in the interval $[1,\zeta(r))$ if and only if $r\leq\eta$. 
\end{theorem} 
\begin{proof}
In virtue of the preceding paragraph, we know from the fact that 
\[g_2(\eta)=F(2,\eta)-\zeta(\eta)=\left(\frac{2^{\eta}}{2^{\eta}-1}\right)\left(\frac{3^{\eta}+1}{3^{\eta}-1}\right)-\zeta(\eta)=0\]
that if $r\in(1,3]$, then the range of $\sigma_{-r}$ is dense in $[1,\zeta(r))$ if and only if $r\leq\eta$. We now show that the range of $\sigma_{-r}$ is not dense in $[1,\zeta(r))$ if $r>3$. To do so, we merely need to show that $F(1,r)>\zeta(r)$ for all $r>3$. For $r>3$, we have 
\[F(1,r)=\left(1+\frac{1}{2^r}\right)\sum_{j=0}^{\infty}\frac{1}{2^{jr}}>\left(1+\frac{1}{2^r}\right)^2=1+\frac{1}{2^r}+\frac{3}{4}\left(\frac{1}{2^{r-1}}\right)\] 
\[>1+\frac{1}{2^r}+\frac{1}{(r-1)2^{r-1}}=1+\frac{1}{2^r}+\int_2^{\infty}\frac{1}{x^r}dx>\zeta(r).\]
\end{proof}
\section{An Open Problem}
We end by acknowledging that it might be of interest to consider the number of ``gaps" in the range of $\sigma_{-r}$ for various $r$. For example, for which values of $r\in(1,\infty)$ is there precisely one gap in the range of $\sigma_{-r}$? More generally, if we are given a positive integer $L$, then, for what values of $r>1$ is the closure of the range of $\sigma_{-r}$ a union of exactly $L$ disjoint subintervals of $[1,\zeta(r)]$?

\bigskip
\hrule
\bigskip

\noindent 2010 {\it Mathematics Subject Classification}:  Primary 11B05; Secondary 11A25.

\noindent \emph{Keywords: } Density, divisor function


\begin{thebibliography}{9}

\bibitem{Dusart10}
Dusart, Pierre. Estimates of some functions over primes without R.H., arXiv:1002.0442
(2010).

\bibitem{Laatsch86} 
Laatsch, Richard. Measuring the abundancy of integers. Math. Mag. 59 (1986), no. 2, 84--92.

\bibitem{Wolfram09}
Wolfram Research, Inc., Mathematica, Version 9.0, Champaign, IL (2012).

\end{thebibliography}
\end{document}